\documentclass[10pt]{amsart}

\usepackage{amsmath, amsthm}
\usepackage{eucal}

\usepackage{amssymb}
\usepackage{amscd}
\usepackage{palatino}
\usepackage{latexsym}
\usepackage{epsfig}
\usepackage{graphicx}
\usepackage{amsfonts}
\usepackage{psfrag}

\usepackage{url}
\usepackage{cite}

\usepackage[margin=1in]{geometry}

\input xy
\xyoption{all}
\UseComputerModernTips

\numberwithin{equation}{section}
\numberwithin{figure}{section}

\newtheorem{theorem}{Theorem}[section]
\newtheorem{lemma}[theorem]{Lemma}

\newtheorem{corollary}[theorem]{Corollary}

\newtheorem{remark}[theorem]{Remark}
\newtheorem{example}[theorem]{Example}

\theoremstyle{definition}

\newcommand{\C}{{\mathbb{C}}}

\newcommand{\Q}{{\mathbb{Q}}}

\newcommand{\into}{\hookrightarrow}

\DeclareMathOperator{\pt}{pt}

\newcommand{\hsm}{{\hspace{1mm}}}

\usepackage{multicol}
\usepackage[dvips]{color}

\definecolor{gold}{rgb}{0.85,.66,0}
\definecolor{cherry}{rgb}{0.9,.1,.2}
\definecolor{burgundy}{rgb}{0.8,.2,.2}
\definecolor{orangered}{rgb}{0.85,.3,0}
\definecolor{orange}{rgb}{0.85,.4,0}
\definecolor{olive}{rgb}{.45,.4,0}
\definecolor{lime}{rgb}{.6,.9,0}
\definecolor{green}{rgb}{.2,.7,0}
\definecolor{grey}{rgb}{.4,.4,.2}
\definecolor{brown}{rgb}{.4,.3,.1}

\newcommand\toe[2]{\mathcal{T}_{#1}(#2)}
\newcommand\head[2]{\mathcal{H}_{#1}(#2)}
\newcommand{\Flags}{{\mathcal{F}\ell ags}}

\begin{document}

\title{A Giambelli formula for the $S^1$-equivariant cohomology of type $A$ Peterson varieties}

\author{Darius Bayegan}
\address{Department of Pure Mathematics and Mathematical 
Statistics\\ Centre for Mathematical Sciences \\ Wilberforce Road \\
Cambridge CB3 0WA \\ United Kingdom}

\author{Megumi Harada}
\address{Department of Mathematics and
Statistics\\ McMaster University\\ 1280 Main Street West\\ Hamilton, Ontario L8S4K1\\ Canada}
 \email{Megumi.Harada@math.mcmaster.ca}
\thanks{The second author is partially supported by an NSERC Discovery Grant,
an NSERC University Faculty Award, and an Ontario Ministry of Research
and Innovation Early Researcher Award.}

\keywords{} 
\subjclass[2000]{Primary: 14N15; Secondary: 55N91}

\date{\today}



\begin{abstract}

  The main result of this note is a \textbf{Giambelli formula} for
  the Peterson Schubert classes in the $S^1$-equivariant cohomology ring
  of a type $A$ Peterson variety. Our results depend on the Monk
  formula for the equivariant structure constants for the Peterson Schubert
  classes derived by Harada and Tymoczko. In addition, we give
  proofs of two facts observed by H. Naruse: firstly, that some
  constants which appear in the multiplicative structure of the
  $S^1$-equivariant cohomology of Peterson varieties 
  are \textbf{Stirling numbers of the second kind}, and secondly, that the
  Peterson Schubert classes satisfy a \textbf{stability property} in a sense analogous to the
  stability of the classical equivariant Schubert classes in the
  $T$-equivariant cohomology of the flag variety. 
\end{abstract}

\maketitle

\setcounter{tocdepth}{1}
\tableofcontents

\section{Introduction}\label{sec:intro}

The main result of this note is a \textbf{Giambelli formula} in the
$S^1$-equivariant cohomology\footnote{In this note, all cohomology rings are with
  coefficients in $\C$.} of type $A$ Peterson
varieties. Specifically, we give an explicit formula which expresses
an arbitrary \textbf{Peterson Schubert class} in terms of the degree-$2$
Peterson Schubert classes. We call this a ``Giambelli formula'' by
analogy with the standard Giambelli formula in classical Schubert
calculus \cite{Ful97} which expresses an arbitrary Schubert class in terms
of degree-$2$ Schubert classes. 

We briefly recall the setting of our results. \textbf{Peterson
 varieties} in type $A$ can be 
defined as the following subvariety $Y$ of
$\mathcal{F}\ell ags(\C^n)$: 
\begin{equation}\label{eq:def intro}
Y := \{ V_\bullet = (0 \subseteq V_1 \subseteq V_2 \subseteq \cdots
\subseteq V_{n-1} \subseteq V_n = \C^n) \hsm \mid \hsm  NV_i \subseteq
V_{i+1} \textup{ for all } i = 1, \ldots, n-1\} 
\end{equation}
where $N: \C^n \to \C^n$ denotes the principal nilpotent operator.
These varieties have been much studied due to its relation to 
the quantum cohomology of the flag variety \cite{Kos96,
  Rie03}. Thus it is natural to study their topology, e.g. the
structure of their (equivariant) cohomology rings. We do so through
Schubert calculus techniques. Our results extend techniques initiated
and developed in \cite{HarTym09, HarTym10}, to which we refer the
reader for further details and motivation.

There is a natural circle subgroup of $U(n,\C)$ which acts
on $Y$ (recalled in Section~\ref{sec:background}).
The inclusion of $Y$
into $\Flags(\C^n)$ induces a natural ring homomorphism
\begin{equation}\label{eq:intro-proj}
H^*_T(\mathcal{F}\ell ags(\C^n)) \to H^*_{S^1}(Y)
\end{equation} 
where $T$ is the subgroup of diagonal matrices of $U(n,\C)$ acting in the
usual way on $\Flags(\C^n)$. One of the main results of \cite{HarTym09}
is that a certain subset of the equivariant Schubert classes
$\{\sigma_w\}_{w \in S_n}$ in $H^*_T(\Flags(\C^n))$ maps under the projection~\eqref{eq:intro-proj} to a
computationally convenient module basis of $H^*_{S^1}(Y)$. We refer to
the images via~\eqref{eq:intro-proj} of $\{\sigma_w\}_{w \in S_n}$ in
$H^*_{S^1}(Y)$ as \textbf{Peterson Schubert classes}. Moreover,
\cite[Theorem 6.12]{HarTym09} gives a manifestly positive \textbf{Monk
  formula} for the product of a degree-$2$ Peterson Schubert class
with an arbitrary Peterson Schubert class, expressed as a
$H^*_{S^1}(\pt)$-linear combination of Peterson Schubert
classes. This is an example of equivariant Schubert calculus in the
realm of Hessenberg varieties (of which Peterson varieties are a
special case), and we view the Giambelli formula (Theorem~\ref{theorem:giambelli}) as a
further development of this theory. The 
Giambelli formula for Peterson varieties was also independently
observed by H. Naruse.

Our Giambelli formula also allows us to simplify the presentation of
the ring $H^*_{S^1}(Y)$ given in \cite[Section 6]{HarTym09}. This is
because the previous presentation used as its generators all of the
elements in the module basis given by Peterson Schubert classes,
although the ring $H^*_{S^1}(Y)$ is multiplicatively generated by only
the degree-$2$ Peterson Schubert classes.  Details are explained in
Section~\ref{subsec:simplify} below, where we also give a concrete
example in $n=4$ to illustrate our results. We also formulate a
conjecture (cf. Remark~\ref{remark:quadratic}), suggested to us by 
the referee of this manuscript, that the ideal of defining relations is in fact
generated by the quadratic relations only. If true, this would be a
significant further simplification of the presentation of this ring and would
lead to interesting further questions (both combinatorial and
geometric).

In Sections~\ref{sec:stirling} and~\ref{sec:stability}, we present
proofs of two facts concerning Peterson Schubert classes which we
learned from H. Naruse. The results are due to Naruse but the proofs
given here are our own. We chose to include these results because
they do not appear elsewhere in the literature.  The first fact is
that \textbf{Stirling numbers of the second kind} (see
Section~\ref{sec:stirling} for the definition) appear naturally in the
product structure of $H^*_{S^1}(Y)$. The second is that the Peterson
Schubert classes satisfy a \textbf{stability condition} with respect
to the natural inclusions of Peterson varieties induced from the
inclusions $\Flags(\C^n) \into \Flags(\C^{n+1})$.

\medskip

\noindent \textbf{Acknowledgements.} We are grateful to Hiroshi Naruse
for communicating to us his observations on the stability of Peterson
Schubert classes and the presence of Stirling numbers of the second
kind. We thank Julianna Tymoczko for support of and interest in this
project and Alex Yong for useful conversations. We also 
thank an anonymous referee for a careful reading of the manuscript and
many helpful suggestions; in particular, the referee suggested the
conjecture regarding quadratic relations recorded in
Remark~\ref{remark:quadratic}, as well as the question which we record
in Remark~\ref{remark:normal crossings}.

\section{Peterson varieties and $S^1$-fixed
  points}\label{sec:background}

In this section we briefly recall the objects under study. For details
we refer the reader to \cite{HarTym09}. Since we work exclusively in
Lie type $A$ we henceforth omit it from our terminology. 

By the \textbf{flag variety} we mean the homogeneous space
$GL(n,\C)/B$ where $B$ is the standard Borel subgroup of upper-triangular
invertible matrices. The flag variety can also be identified with the
space of nested subspaces in $\C^n$, i.e., 
\[
\mathcal{F}\ell ags(\C^n) := 
\{ V_{\bullet} = (\{0\} \subseteq V_1 \subseteq V_2 \subseteq \cdots \subseteq V_{n-1} \subseteq
V_n = \C^n) \hsm \mid \hsm \dim_{\C}(V_i) = i \} \cong GL(n,\C)/B. 
\]
Let $N$ be the $n \times n$ principal nilpotent operator
given with respect to the standard basis of $\C^n$ as the matrix
with one $n \times n$ Jordan block of eigenvalue $0$, i.e., 
\begin{equation}\label{eq:standard principal nilpotent}
N = \begin{bmatrix} 0 & 1 & 0 &         & & \\
                    0 & 0 & 1 &         & & \\
                    0 & 0 & 0 &         & & \\
                      &   &   & \ddots  & & \\
                      &   &   &         & 0 & 1\\
                      &   &   &         & 0 & 0 \\
    \end{bmatrix}.
\end{equation}
Fix $n$ a positive integer. The main geometric object under study, the
\textbf{Peterson variety} $Y$, is the subvariety of $\mathcal{F}\ell
ags(\C^n)$ defined in~\eqref{eq:def intro} 
where $N$ is the standard principal nilpotent in~\eqref{eq:standard
  principal nilpotent}. The variety $Y$ is a (singular) projective
variety of complex dimension $n-1$.

We recall some facts from \cite{HarTym09}. 
The following circle subgroup of $U(n,\C)$ preserves $Y$: 
\begin{equation}\label{eq:def-circle}
S^1  =   \left\{ \left. \begin{bmatrix} t^n & 0 & \cdots & 0 \\ 0 & t^{n-1} &  &
      0 \\ 0 & 0 & \ddots & 0 \\ 0 & 0 &  & t \end{bmatrix}
  \; \right\rvert \;  t \in \C, \; \|t\| = 1 \right\}  \subseteq T^n
\subseteq U(n,\C). 
\end{equation}
Here $T^n$ is the standard maximal torus of $U(n,\C)$ consisting of diagonal unitary matrices.
The $S^1$-fixed points of $Y$ are isolated and are
a subset of the $T^n$-fixed points of $\Flags(\C^n)$.
As is standard, we identify the $T^n$-fixed points in
$\Flags(\C^n)$ with the permutations $S_n$.
In particular since
$Y^{S^1}$ is a subset of $\Flags(\C^n)^{T^n}$, we think of the
Peterson fixed points as permutations in $S_n$. 
There is a natural bijective correspondence between the Peterson fixed
points $Y^{S^1}$ and subsets of $\{1,2,\ldots,n-1\}$ 
which we now
briefly recall. It is explained in
\cite[Section 2.3]{HarTym09} that a permutation $w \in S_n$ is in
$Y^{S^1}$ precisely when the one-line notation of $w^{-1}$ is of the
form 
\begin{equation}\label{eq:w-oneline}
w^{-1} = \underbrace{j_1 \,  j_1 -1 \, \cdots \, 1}_{j_1 \textup{ entries}} \,  \underbrace{j_2 \,  j_2-1 \, \cdots \,
j_1+1}_{j_2-j_1 \textup{ entries}} \, \cdots \,
\underbrace{n \, n-1 \, \cdots \, j_m+1}_{n-j_m \textup{ entries}}
\end{equation}
where \(1 \leq j_1 < j_2 < \cdots < j_m < n\) is any sequence of
strictly increasing integers. For example, for $n=9, m=2$ and $j_1 = 3, j_2 = 7$, then the permutation $w^{-1}$ in~\eqref{eq:w-oneline} has one-line notation $321765498$. Thus for each permutation \(w \in S_n\)
satisfying~\eqref{eq:w-oneline} we define 
\[
{\mathcal{A}} := \{ i: w^{-1}(i) = w^{-1}(i+1)+1 \textup{ for } 1 \leq i\leq n-1\} \subseteq
\{1,2,\ldots, n-1\}.
\]
This gives a one-to-one correspondence between the power set of
$\{1,2,\ldots, n-1\}$ and $Y^{S^1}$. We denote the Peterson fixed point
corresponding to a subset $\mathcal{A} \subseteq \{1,2,\ldots, n-1\}$
by $w_{\mathcal{A}}$.

\begin{example}
Let $n=5$ and suppose $\mathcal{A}=\{1,2,4\}$. Then the associated
permutation is 
$w_{\mathcal{A}} = 32154$.
\end{example}

Indeed, for a fixed $n$, we can also easily enumerate all the Peterson fixed points by using
this correspondence. 

\begin{example}\label{example:Peterson Variety}
Let $n=4$. Then $Y^{S^1}$ consists of $2^3 = 8$ elements in
correspondence with the subsets of $\{1,2,3\}$, namely: 
$w_{\emptyset} = 1234, w_{\{1\}} = 2134, w_{\{2\}} = 1324, w_{\{3\}} =
1243, w_{\{1,2\}} = 3214, w_{\{2,3\}} = 1432, w_{\{1,3\}} = 2143,
w_{\{1,2,3\}} = 4321$. 
\end{example}

Given a choice of subset $\mathcal{A} \subseteq \{1,2,\ldots,n-1\}$,
there is a natural decomposition of $\mathcal{A}$ as follows. 
We say that a set of consecutive integers
\[\{a, a+1, \ldots, a+k\} \subseteq {\mathcal{A}}\] is a {\bf maximal
  consecutive (sub)string} of ${\mathcal{A}}$ if $a$ and $k$ are such
that neither $a-1$ nor $a+k+1$ is in ${\mathcal{A}}$.  For $a_1 := a$
and $a_2 := a_1+k$, we denote the corresponding maximal consecutive
substring by $[a_1,a_2]$.  It is straightforward to see that any
${\mathcal{A}}$ uniquely decomposes into a disjoint union of maximal
consecutive substrings
\[ {\mathcal{A}} = [a_1, a_2] \cup [a_3,a_4] \cup \cdots \cup [a_{m-1},a_m].\] 
For instance, if $\mathcal{A} = \{1,2,3, 5,6,8\}$, then its
decomposition into maximal
consecutive substrings is $\{1,2,3\} \cup \{5,6\} \cup \{8\} = [1,3]
\cup [5,6] \cup [8,8]$.

Suppose $\mathcal{A} = \{j_1 < j_2 < \cdots < j_m\}$. Finally we
recall that we can associate to each $w_{\mathcal{A}}$ a 
permutation $v_{\mathcal{A}}$ by the recipe
\begin{equation}\label{eq:def vA}
w_{\mathcal{A}} \mapsto v_{\mathcal{A}} := s_{j_1} s_{j_2} \cdots s_{j_m}
\end{equation}
where an $s_i$ denotes the simple transposition $(i,i+1)$ in $S_n$.

\section{The Giambelli formula for Peterson varieties}

\subsection{The Giambelli formula}

In this section we prove the main result of this note, namely, a
\textbf{Giambelli formula for Peterson varieties}.

As recalled above, the Peterson variety $Y$ is an 
$S^1$-space for a subtorus $S^1$ of $T^n$ and it can be checked that 
\(Y^{S^1} = (\Flags(\C^n))^{T^n} \cap Y.\) There is a
forgetful map from $T^n$-equivariant cohomology to
$S^1$-equivariant cohomology obtained by the inclusion 
\(S^1 \into T\), so there is a commutative diagram 
\begin{equation}\label{eq:comm-diag}
\xymatrix{
H^*_{T^n}(\Flags(\C^n)) \ar[r] \ar[d] & H^*_{T^n}((\Flags(\C^n))^{T^n}) \ar[d] \\
H^*_{S^1}(\Flags(\C^n)) \ar[r] \ar[d] & H^*_{S^1}((\Flags(\C^n))^{T^n}) \ar[d] \\
H^*_{S^1}(Y) \ar[r] & H^*_{S^1}(Y^{S^1}). 
}
\end{equation}
The \textbf{equivariant Schubert classes} $\{\sigma_w\}$ in
$H^*_{T^n}(\Flags(\C^n))$ are well-known to form a~$H^*_{T^n}(\pt)$-module~basis for $H^*_{T^n}(\Flags(\C^n))$. 
We call the image of $\sigma_w$ under the projection map $H^*_{T^n}(\Flags(\C^n))
\to H^*_{S^1}(Y)$ the \textbf{Peterson
    Schubert class corresponding to $w$}. For the permutations
  $v_{\mathcal{A}}$ defined in~\eqref{eq:def vA}, we denote by
  $p_{\mathcal{A}}$ the
  corresponding Peterson Schubert class, i.e. the image of
  $\sigma_{v_{\mathcal{A}}}$. (This is slightly different notation
  from that used in \cite{HarTym09}.) 
We denote by
$p_{\mathcal{A}}(w) \in H^*_{S^1}(\pt) \cong \C[t]$ the restriction of the
Peterson Schubert class $p_{\mathcal{A}}$ to the fixed point $w \in Y^{S^1}$.

One of the main results of \cite{HarTym09} is that the set of
$2^{n-1}$ Peterson
Schubert classes \(\{p_{\mathcal{A}}\}_{\mathcal{A} \subseteq
  \{1,2,\ldots, n-1\}}\) form a $H^*_{S^1}(\pt)$-module basis for
$H^*_{S^1}(Y)$ where $v_{\mathcal{A}}$ is defined in~\eqref{eq:def
  vA}. (The fact that $H^*_{S^1}(Y)$ is a free module of rank $2^{n-1}$
over $H^*_{S^1}(\pt)$ fits nicely with the result \cite[Theorem
10.2]{SommersTymoczko} that the Poincar\'e polynomial of $Y$ is given
by $(q^2+1)^{n-1}$.) It is also shown in \cite{HarTym09} that the $n-1$
degree-$2$ classes
$\{p_i := p_{s_i}\}_{i=1}^{n-1}$ form a multiplicative set of
generators for $H^*_{S^1}(Y)$. These classes $p_i$ are also
(equivariant) Chern classes of certain line bundles over $Y$. Moreover, there is a \textbf{Monk formula} \cite[Theorem 6.12]{HarTym09}  
which expresses a product 
\[
p_i p_{\mathcal{A}}
\]
for any $i \in \{1,2,\ldots,n-1\}$ and any $\mathcal{A} \subseteq
\{1,2,\ldots,n-1\}$ as a $H^*_{S^1}(\pt)$-linear combination
of the additive module basis
$\{p_{\mathcal{A}}\}$. Since the
$p_i$ multiplicatively generate the ring, this Monk formula completely
determines the ring structure of $H^*_{S^1}(Y)$.  
Furthermore it is in principle possible to express any
$p_{\mathcal{A}}$ in terms of the $p_i$. Our Giambelli formula is an
explicit formula which achieves this (cf. for example \cite{Ful97} for
the version in classical Schubert calculus).

We begin by recalling the Monk formula, for which we need some terminology. 
Fix ${\mathcal{A}} \subseteq \{1,2,\ldots,n-1\}$.  We
define 
$\mathcal{H}_{\mathcal{A}}: \mathcal{A} 
\to \mathcal{A}$ by 
\[\head{{\mathcal{A}}}{j} = \textup{the maximal element  in the maximal consecutive substring of ${\mathcal{A}}$ containing $j$}.\]
Similarly, we define 
$\mathcal{T}_{\mathcal{A}}: \mathcal{A}
\rightarrow \mathcal{A}$ by 
\[\toe{{\mathcal{A}}}{j} = 
\textup{the minimal element in the maximal consecutive substring of
  ${\mathcal{A}}$ containing } j.\] We say that the maps
$\mathcal{H}_{\mathcal{A}}$ and $\mathcal{T}_{\mathcal{A}}$ give the
``head" and ``tail" of each maximal consecutive substring of
$\mathcal{A}$. For an example see \cite[Example 5.6]{HarTym09}.
We recall the following. 

\begin{theorem}\label{theorem:Monk} \textbf{``The Monk formula for Peterson varieties'' (\cite[Theorem 6.12]{HarTym09})}
Fix $n$ a positive integer. Let $Y$ be the Peterson variety in $\Flags(\C^n)$ with the 
natural $S^1$-action defined by~\eqref{eq:def-circle}. For ${\mathcal A}
\subseteq \{1,2,\ldots, n-1\}$, let $v_{\mathcal{A}} \in S_n$ be the permutation
in~\eqref{eq:def vA}, and let $p_{\mathcal{A}}$ be the
corresponding Peterson Schubert class in $H^*_{S^1}(Y)$. Then 
\begin{equation}\label{eq:Monk-final}
p_i \cdot p_{\mathcal{A}} = p_i(w_{\mathcal{A}}) \cdot 
p_{\mathcal{A}} +  \sum_{{\mathcal{A}} \subsetneq {\mathcal{B}} \textup{ and } |{\mathcal{B}}|
  = |{\mathcal{A}}|+1} c^{\mathcal{B}}_{i,{\mathcal{A}}} \cdot p_{\mathcal{B}},
\end{equation}
where, for a subset $\mathcal{B} \subseteq \{1,2,\ldots,n-1\}$ which is a disjoint union
\(\mathcal{B} =
  \mathcal{A} \cup \{k\},\) 
\begin{itemize} 
\item if \(i \not \in \mathcal{B}\) then
  $c^{\mathcal{B}}_{i,\mathcal{A}} = 0$, 
\item if \(i \in {\mathcal{B}}\) and \(i \not \in [\toe{\mathcal{B}}{k},
\head{\mathcal{B}}{k}],\) then $c^{\mathcal{B}}_{i,{\mathcal{A}}} =
0$, 
\item if \(k \leq i \leq \head{\mathcal{B}}{k},\) 
  then 
\begin{equation}\label{eq:cBiA-formula-part1}
c^{\mathcal{B}}_{i,{\mathcal{A}}} = (\head{{\mathcal{B}}}{k}-i+1) \cdot \left( \begin{array}{c} \head{{\mathcal{B}}}{k} - \toe{{\mathcal{B}}}{k}+1 \\
    k-\toe{{\mathcal{B}}}{k} \end{array} \right),
\end{equation}
\item if \(\toe{{\mathcal{B}}}{k} \leq i \leq k-1,\) 
\begin{equation}\label{eq:cBiA-formula-part2}
c^{\mathcal{B}}_{i,{\mathcal{A}}} 
 = (i-\toe{{\mathcal{B}}}{k}+1) \cdot \binom{\head{{\mathcal{B}}}{k}-\toe{{\mathcal{B}}}{k}+1}{k-\toe{{\mathcal{B}}}{k}+1}.
\end{equation}
\end{itemize} 
\end{theorem}

We also recall that \cite[Lemma 6.7]{HarTym09} implies that 
if $\mathcal{B}, \mathcal{B}'$ are two disjoint subsets of
$\{1,2,\ldots,n-1\}$ such that 
there is no $i$ in $\mathcal{B}$ and $j$ in $\mathcal{B}'$ 
with $\lvert {i-j} \rvert=1$, then $p_{\mathcal{B}\cup\mathcal{B}'} =p_{\mathcal{B}}p_{\mathcal{B}'}$. 
It follows that 
for any $\mathcal{A}$ we have 
\begin{equation}\label{eq:product pvA}
p_{\mathcal{A}} = p_{[a_1, a_2]} \cdot p_{[a_3,a_4]} \cdots p_{[a_{m-1},a_m]}
\end{equation}
where \({\mathcal{A}} = [a_1, a_2] \cup [a_3,a_4] \cup \cdots \cup
[a_{m-1},a_m]\) is the decomposition of $\mathcal{A}$ into maximal
consecutive substrings. In particular, in order to give an expression
for $p_{\mathcal{A}}$ in terms of the elements $p_i$,
from~\eqref{eq:product pvA} we see that it suffices
to give a formula only for the special case in which $\mathcal{A}$
consists of a single maximal consecutive string.

We now state and prove our Giambelli formula.

\begin{theorem}\label{theorem:giambelli} 
Fix $n$ a positive integer. Let $Y$ be the Peterson variety in $\mathcal{F}\ell ags(\C^n)$
with the $S^1$-action defined by~\eqref{eq:def-circle}. Suppose
$\mathcal{A}=\{a,a+1,a+2, \ldots, a+k\}$ where $1 \leq a \leq n-1$ and
$0 \leq k \leq n-1-a$. Let $v_{\mathcal{A}}$ be the permutation
corresponding to $\mathcal{A}$ 
defined in~\eqref{eq:def vA} and let $p_{\mathcal{A}}$ be the
associated Peterson
Schubert class. Then
\[
p_{\mathcal{A}} =\frac{1}{(k+1)!} \displaystyle\prod_{j \in
  \mathcal{A}} p_{j}.
\]
\end{theorem}

We use the following
lemma.

\begin{lemma}\label{lemma:simple monk if i not in A} 
Suppose $i \in \{1,2,\ldots,n-1\}$ and $\mathcal{A} \subseteq
\{1,2,\ldots,n-1\}$. Suppose further that $i \not \in
\mathcal{A}$. Then the Monk relation 
\[
p_{i} \cdot p_{\mathcal{A}} = p_{i}(w_{\mathcal{A}}) \cdot 
p_{\mathcal{A}} + \displaystyle\sum_{\mathcal{A} \subset
    {\mathcal{B}} \textup{ and }
    |{\mathcal{B}}|=|\mathcal{A}|+1} c^{{\mathcal{B}}}_{i,\mathcal{A}}
  \cdot p_{{\mathcal{B}}}
\]
simplifies to 
\begin{equation}\label{eq:final simple monk}
p_{i} \cdot p_{\mathcal{A}} =  c^{{\mathcal{A} \cup \{i\}}}_{i,\mathcal{A}}
  \cdot p_{{\mathcal{A} \cup \{i\}}}.
\end{equation}
\end{lemma} 

\begin{proof} 
First observe that the Monk relation 
simplifies to 
\begin{equation}\label{eq:simplified monk}
p_{i} \cdot p_{\mathcal{A}} = \displaystyle\sum_{\mathcal{A} \subset
    {\mathcal{B}} \textup{ and }
    |{\mathcal{B}}|=|\mathcal{A}|+1} c^{{\mathcal{B}}}_{i,\mathcal{A}}
  \cdot p_{{\mathcal{B}}}
\end{equation}
if $i \not \in \mathcal{A}$, since in that case
$p_i(w_{\mathcal{A}}) = 0$ by \cite[Lemma 6.4]{HarTym09}. 
Moreover, from Theorem~\ref{theorem:Monk} we also know that $c_{i,
  \mathcal{A}}^{\mathcal{B}} = 0$ if $i \not \in \mathcal{B}$. 
Hence the summands appearing in~\eqref{eq:simplified monk} correspond to $\mathcal{B}$ satisfying
$\mathcal{A} \subseteq \mathcal{B}, \lvert \mathcal{B} \rvert = \lvert
\mathcal{A} \rvert +1$, and $i \in \mathcal{B}$. On the other hand,
since $i \not \in \mathcal{A}$ by assumption, this means that there is
only one non-zero summand in the right hand side
of~\eqref{eq:simplified monk}, namely, the term corresponding to
$\mathcal{B} = \mathcal{A} \cup \{i\}$. 
The equation~\eqref{eq:final simple monk} follows. 
\end{proof} 

We now prove the main theorem. 

\begin{proof}[Proof of Theorem~\ref{theorem:giambelli}]
  We proceed by induction on $k$. First consider the base case
  where $k=0$. Then $A=\{a\}$, so $p_{v_\mathcal{A}}=p_a$. On the right hand
  side, we have $\frac{1}{(0+1)!}\prod_{j \in \mathcal{A}}p_{j}=p_a$.
This verifies
  the base case.

  By induction, suppose the claim holds for
  $k-1$. We now show that the claim holds for
  $k$. Consider $\mathcal{A}' := \{a,a+1,\ldots,a+k-1\}$ and 
consider the product 
  $p_{a+k} \cdot p_{\mathcal{A}'}$. 
From the Monk
  formula in Theorem~\ref{theorem:Monk} we know
  that 
\begin{equation}\label{eq:expansion}
p_{a+k} \cdot p_{\mathcal{A}'}
=
p_{a+k}(w_{\mathcal{A}'}) \cdot p_{\mathcal{A}'}+\displaystyle\sum_{{\mathcal{A}}'\subseteq
  {\mathcal{B}} \textup{ and } |{\mathcal{B}}|=|{\mathcal{A}}'|+1}c^{{\mathcal{B}}}_{a+k,{\mathcal{A}}'} \cdot p_{\mathcal{B}}.
\end{equation}
On the other hand since by definition $a+k \not \in \mathcal{A}'$, by Lemma~\ref{lemma:simple monk if i not in A} 
the equality~\eqref{eq:expansion} further simplifies to
\[
 p_{a+k} \cdot p_{\mathcal{A}'}=c^{{\mathcal{A}}}_{a+k,{\mathcal{A}}'} \cdot p_{\mathcal{A}}.
\]
Moreover, 
since ${\mathcal{A}}={\mathcal{A}}'\cup \{a+k\}$, we have
$\mathcal{H}_{{\mathcal{A}}}(a+k) = a+k$ and $\toe{\mathcal{A}}{a+k} =
a$. Hence by Theorem~\ref{theorem:Monk}
\begin{equation}
\begin{split}
c^{{\mathcal{A}}}_{a+k,{\mathcal{A}}'} & =\left( \mathcal{H}_{{\mathcal{A}}}(a+k)-(a+k)+1 \right)
\binom{\mathcal{H}_{{\mathcal{A}}}(a+k)-\mathcal{T}_{{\mathcal{A}}}(a+k)+1}{(a+k)-\mathcal{T}_{\mathcal{A}}(a+k)} \\
& = ((a+k)-(a+k)+1) \binom{a+k-a+1}{(a+k)-a} \\
& = k+1.
\end{split}
\end{equation}
Therefore 
\[
p_{a+k} \cdot p_{\mathcal{A}'}=(k+1) \cdot p_{\mathcal{A}}.
\]
By the inductive hypothesis we have for the set
$\mathcal{A}' = \{a, a+1, \ldots, a+k-1\}$ 
\[
p_{\mathcal{A}'} =\frac{1}{k!} \prod_{j \in
  {\mathcal{A}'}}p_{j}.
\]
Substituting into the above equation yields 
\[
p_{\mathcal{A}} = \frac{1}{(k+1)!} \prod_{j \in {\mathcal{A}}}
p_{j} 
\]
as desired. This completes the proof. 
\end{proof}

\begin{remark}\label{remark:normal crossings}
  We thank the referee for the following observation. The
  formula in Theorem~\ref{theorem:giambelli} suggests that the classes
  $p_i$ behave like a normal crossings divisor (up to quotient
  singularities), with all other classes arising (up to rational
  coefficients) as intersections of the components. It would certainly
  be of
  interest to understand more precisely the underlying geometry which
  gives rise not only to the Giambelli relation in
  Theorem~\ref{theorem:giambelli} but also to the original Monk formula
  \cite[Theorem 6.12]{HarTym09}. 
\end{remark}

From Theorem~\ref{theorem:giambelli} it immediately follows that for
any subset
\[ {\mathcal{A}} = [a_1, a_2] \cup [a_3,a_4] \cup \cdots \cup
[a_{m-1},a_m]\] 
with its decomposition into maximal consecutive substrings, we have 
\begin{equation}\label{eq:giambelli for general pA} 
p_{\mathcal{A}} = \frac{1}{(a_2-a_1+1)!} \cdot
\frac{1}{(a_4-a_3+1)!} \cdots \frac{1}{(a_m-a_{m-1}+1)!} \prod_{j \in
  \mathcal{A}} p_j.
\end{equation}
For the purposes of the next section we introduce the
notation 
\begin{equation}\label{eq:def sigma}
\sigma(\mathcal{A}) := \frac{1}{(a_2-a_1+1)!} \cdot
\frac{1}{(a_4-a_3+1)!} \cdots \frac{1}{(a_m-a_{m-1}+1)!}
\end{equation}
for the rational coefficient appearing in~\eqref{eq:giambelli for
  general pA}. The following is an immediate corollary of this
discussion.

\begin{corollary} 
Let 
\[ {\mathcal{A}} = [a_1, a_2] \cup [a_3,a_4] \cup \cdots \cup
[a_{m-1},a_m].\] 
Then 
\[p_{\mathcal{A}} = \sigma(\mathcal{A}) \prod_{j \in \mathcal{A}}
p_j.
\]
\end{corollary}

\subsection{Simplification of  the Monk relations}\label{subsec:simplify}

In this section we explain how to use the Giambelli formula to
simplify the ring presentation of $H^*_{S^1}(Y)$ given in 
\cite[Section 6]{HarTym09}.
Recall that 
the Peterson Schubert classes $\{p_{\mathcal{A}}\}$ form an additive
module basis for $H^*_{S^1}(Y)$ and the degree $2$ classes
$\{p_i\}_{i=1}^{n-1}$ form a multiplicative basis, so the Monk
relations give a 
presentation of the ring $H^*_{S^1}(Y)$ via generators and
relations as follows.

\begin{theorem}\label{theorem:presentation} \textbf{(\cite[Corollary
    6.14]{HarTym09})} 
Fix $n$ a positive integer. Let $Y$ be
  the Peterson variety in $\Flags(\C^n)$ with the $S^1$-action defined by~\eqref{eq:def-circle}. For ${\mathcal A} \subseteq
  \{1,2,\ldots, n-1\}$, let $v_{\mathcal{A}} \in S_n$ be the
  permutation given in~\eqref{eq:def vA}, and let
  $p_{\mathcal{A}}$ be the corresponding Peterson Schubert class
  in $H^*_{S^1}(Y)$.  Let \(t \in H^*_{S^1}(\pt) \cong \C[t]\) denote
  both the generator of $H^*_{S^1}(\pt)$ and its image \(t \in
  H^*_{S^1}(Y).\) Then the $S^1$-equivariant cohomology $H^*_{S^1}(Y)$
  is given by
\[
H^*_{S^1}(Y) \cong \C[t, \{p_{\mathcal{A}}\}_{\mathcal{A}
  \subseteq \{1,2,\ldots, n-1\}}] \large/ \mathcal{J}
\]
where $\mathcal{J}$ is the ideal generated by the
relations~\eqref{eq:Monk-final}.  
\end{theorem}

In order to state the main result of this section we introduce some
notation. For $i$ with $1 \leq i \leq n-1$ and $\mathcal{A} \subseteq
\{1,2,\ldots,n-1\}$ define 
\[
m_{i,\mathcal{A}} := 
p_i \cdot p_{\mathcal{A}}  - p_i(w_{\mathcal{A}}) \cdot 
p_{\mathcal{A}} -  \sum_{{\mathcal{A}} \subsetneq {\mathcal{B}} \textup{ and } |{\mathcal{B}}|
  = |{\mathcal{A}}|+1} c^{\mathcal{B}}_{i,{\mathcal{A}}} \cdot
p_{\mathcal{B}} 
\]
thought of as an element in $\C[t, \{p_{\mathcal{A}}\}_{\mathcal{A}
  \subseteq \{1,2,\ldots,n-1\}}]$ 
where the $c_{i,\mathcal{A}}^{\mathcal{B}} \in \C[t]$ are the
coefficients computed in Theorem~\ref{theorem:Monk}. Motivated by the
Giambelli formula we also define the following elements in $\C[t, p_1,
p_2,\ldots, p_{n-1}]$:
\[
q_{i,\mathcal{A}} := 
p_i \cdot \sigma(\mathcal{A}) \cdot \left( \prod_{j \in \mathcal{A}} p_j \right) -
p_i(w_{\mathcal{A}}) \cdot 
\sigma(\mathcal{A}) \cdot \left( \prod_{j \in \mathcal{A}} p_j \right)
 -  \sum_{{\mathcal{A}} \subsetneq {\mathcal{B}} \textup{ and } |{\mathcal{B}}|
  = |{\mathcal{A}}|+1} c^{\mathcal{B}}_{i,{\mathcal{A}}} \cdot
\sigma(\mathcal{B}) \left( \prod_{k \in \mathcal{B}} p_k \right) 
\]
where the $\sigma(\mathcal{A}) \in \Q$ is the constant defined
in~\eqref{eq:def sigma}.

\begin{example}
Let $n=4$ and $i=1$ and $\mathcal{A} = \{1,2\}$. Consider 
\[
m_{1, \{1,2\}} = p_1 p_{v_{\{1,2\}}} - 2t \, p_{v_{\{1,2\}}} + p_{v_{\{1,2,3\}}}. 
\]
Expanding 
in terms of the Giambelli formula, we obtain 
\[
q_{1, \{1,2\}} = \frac{1}{2} p_{1}^2p_{2} - 
2t \cdot \left( \frac{1}{2} p_{1}p_{2} \right) +\frac{1}{6} 
p_{1}p_{2}p_3 = t \, p_1 p_2 + \frac{1}{6} p_1 p_2 p_3. 
\]
\end{example}

The main theorem of this section gives a ring
presentation of $H^*_{S^1}(Y)$ using fewer generators and fewer
relations than that in Theorem~\ref{theorem:presentation}. More specifically let
$\mathcal{K}$ denote the ideal in $\C[t,p_1, \ldots, p_{n-1}]$
generated by the $q_{i,\mathcal{A}}$ for which $i \in
\mathcal{A}$, i.e., 
\begin{equation}\label{eq:def K} 
\mathcal{K} := \bigg\langle q_{i,\mathcal{A}} \, \bigg\vert \, 1 \leq i \leq
n-1, \mathcal{A} \subseteq \{1,2,\ldots, n-1\}, i \in \mathcal{A}
\bigg\rangle \subseteq \C[t,p_1, \ldots, p_{n-1}].
\end{equation}
We have the following. 

\begin{theorem}\label{theorem:simplify Monk}
Fix $n$ a positive integer. Let $Y$ be the Peterson variety in $\Flags(\C^n)$ equipped with the
action of the $S^1$ in~\eqref{eq:def-circle}.
Then the $S^1$-equivariant cohomology $H^*_{S^1}(Y)$ is isomorphic to the ring 
\[
\C[t, p_1, p_2, \ldots, p_{n-1}]/\mathcal{K}
\]
where $\mathcal{K}$ is the ideal in~\eqref{eq:def K}. 
\end{theorem}

To prove the theorem we need the following lemma. 

\begin{lemma}\label{lemma:some q are zero} 
Let $i \in \{1,2,\ldots,n-1\}$ and $\mathcal{A}
\subseteq \{1,2,\ldots, n-1\}$. Suppose $i \not \in \mathcal{A}$. Then
$q_{i, \mathcal{A}} = 0$ in $\C[t, p_1, p_2, \ldots, p_{n-1}]$. 
\end{lemma}

\begin{proof}
Since $i \not \in \mathcal{A}$ by assumption, Lemma~\ref{lemma:simple
  monk if i not in A} implies that  
\[
m_{i, \mathcal{A}} = p_{i} \cdot p_{\mathcal{A}} - p_{i}(w_{\mathcal{A}})\cdot
p_{\mathcal{A}} - \displaystyle\sum_{\mathcal{A} \subset
    {\mathcal{B}} \textup{ and }
    |{\mathcal{B}}|=|\mathcal{A}|+1} c^{{\mathcal{B}}}_{i,\mathcal{A}}
  \cdot p_{{\mathcal{B}}}
\]
simplifies to 
\begin{equation}
m_{i, \mathcal{A}} = p_{i} \cdot p_{\mathcal{A}} -  c^{{\mathcal{A} \cup \{i\}}}_{i,\mathcal{A}}
  \cdot p_{{\mathcal{A} \cup \{i\}}}.
\end{equation}
Thus in order to compute the corresponding $q_{i,\mathcal{A}}$ 
it remains to compute $c^{{\mathcal{A} \cup \{i\}}}_{i,\mathcal{A}}$
and apply the Giambelli formula.

Let \({\mathcal{A}} = [a_1, a_2] \cup [a_3,a_4] \cup \cdots \cup
[a_{m-1},a_m]\) be the decomposition of $\mathcal{A}$ into maximal
consecutive substrings. Consider the decomposition of 
$\mathcal{A} \cup \{i\}$ into maximal
consecutive substrings. There are several cases to consider:
\begin{enumerate} 
\item the singleton set $\{i\}$ is a maximal consecutive substring of
  $\mathcal{A} \cup \{i\}$, i.e. $i-1 \not \in \mathcal{A}$ and $i+1
  \not \in \mathcal{A}$, 
\item the inclusion of $i$ extends a maximal consecutive substring to its right
  by $1$ element,
  i.e., there exists a maximal consecutive string $[a_\ell, a_{\ell+1}]
  \subseteq \mathcal{A}$ such that $i = a_{\ell+1}+1$ and that 
  $[a_\ell, i]$ is a maximal consecutive substring of $\mathcal{A}
  \cup \{i\}$, 
\item the inclusion of $i$ extends a maximal consecutive substring to
  its left 
  by $1$ element,
  i.e., there exists a maximal consecutive string $[a_\ell, a_{\ell+1}]
  \subseteq \mathcal{A}$ such that $i = a_{\ell}-1$ and that 
  $[i, a_{\ell+1}]$ is a maximal consecutive substring of $\mathcal{A}
  \cup \{i\}$, or 
\item the inclusion of $i$ glues together two maximal consecutive substrings of
  $\mathcal{A}$, i.e., there exist two maximal consecutive substrings
  $[a_{\ell}, a_{\ell+1}], [a_{\ell+2}, a_{\ell+3}]$ such that
  $i=a_{\ell+1}+1 = a_{\ell+2}-1$ and hence $[a_\ell, a_{\ell+3}] =
  [a_{\ell}, a_{\ell+1}] \cup \{i\} \cup [a_{\ell+2}, a_{\ell+3}]$ is
  a maximal consecutive substring of $\mathcal{A} \cup \{i\}$. 
\end{enumerate} 
We consider each case separately. 

Case (1): Suppose $\{i\}$ is a maximal consecutive substring in $\mathcal{A} \cup \{i\}$. 
 In this case, the coefficient $c^{{\mathcal{A} \cup
      \{i\}}}_{i,\mathcal{A}}$ is $1$. 
Hence we have 
\[
m_{i, \mathcal{A}} = p_{i} \, p_{\mathcal{A}} - p_{v_{\mathcal{A} \cup \{i\}}}.
\]
Since $\{i\}$ is a maximal consecutive substring in $\mathcal{A} \cup \{i\}$, we have $\sigma(\mathcal{A}) = \sigma(\mathcal{A} \cup \{i\})$. We conclude
\[
q_{i, \mathcal{A}} = p_i \cdot  \left( \sigma(\mathcal{A}) \cdot \left( \prod_{j \in \mathcal{A}} p_j \right) \right)
 -  \sigma(\mathcal{A} \cup \{i\}) \cdot \left( \prod_{j \in \mathcal{A} \cup \{i\}}
    p_j \right) = 0 
\]
as desired. 

Cases (2) and (3) are very similar, so we only present the argument
for case (2). Suppose $i$ extends a maximal consecutive substring
$[a_\ell, a_{\ell+1}]$ of
$\mathcal{A}$ to its right. 
Then 
\[
m_{i, \mathcal{A}} = p_i \cdot p_{\mathcal{A}} - (i - a_\ell + 1) p_{v_{\mathcal{A} \cup \{i\}}}
\]
since $k=i=\head{\mathcal{B}}{i}$ and $\toe{\mathcal{B}}{i} =
a_\ell$ so $c_{i, \mathcal{A}}^{\mathcal{A} \cup \{i\}} =
i-a_\ell+1$. 
We compute
\begin{equation*}
\begin{split}
q_{i, \mathcal{A}} & = p_i \, \left( \left( \prod_{1 \leq s \leq m-1, \, s \textup{ odd}} \frac{1}{(a_{s+1}-a_s
      +1)!} \right) \cdot \left( \prod_{j \in \mathcal{A}} p_j \right)
\right) \\ 
 & - 
(i - a_\ell+1) \cdot \left( \prod_{1 \leq s \leq m-1, \, s \textup{ odd and } s\neq \ell} \frac{1}{(a_{s+1}-a_s
      +1)!} \right) \cdot \left(\frac{1}{(i - a_\ell + 1)!} \right)
  \cdot \left( \prod_{j \in \mathcal{A} \cup \{i\}} p_j \right)
\end{split}
\end{equation*}
where one of the factors in the product in the second expression has changed because the
maximal consecutive string $[a_\ell, a_{\ell+1}]$ has been extended in 
$\mathcal{A} \cup \{i\}$. Since 
\[
(i - a_\ell+1) \left(\frac{1}{(i - a_\ell + 1)!} \right) =
\frac{1}{(a_{\ell+1}-a_\ell+1)!} 
\]
by assumption on $i$, we conclude $q_{i,\mathcal{A}} = 0$ 
as desired.

Finally, consider the case (4) in which the inclusion of $i$ glues
together two maximal consecutive substrings $[a_\ell, a_{\ell+1}],
[a_{\ell+2}, a_{\ell+3}]$ in $\mathcal{A}$. 
In this case, \(k=i, \head{\mathcal{B}}{i} = a_{\ell+3},
\toe{\mathcal{B}}{i} = a_\ell.\) Hence the coefficient $c_{i,
  \mathcal{A}}^{\mathcal{A} \cup \{i\}}$ is 
\[
c_{i,
  \mathcal{A}}^{\mathcal{A} \cup \{i\}} = (a_{\ell+3} - i+1)
\binom{a_{\ell+3}-a_\ell+1}{i-a_\ell} =
\frac{(a_{\ell+3}-a_\ell+1)!}{(i-a_\ell)! (a_{\ell+3}-i)!}.
\]
The expansion of $p_i \cdot p_{\mathcal{A}}$ is 
the same as in the previous cases. 
The term corresponding to $c^{{\mathcal{A} \cup \{i\}}}_{i,\mathcal{A}}
  \cdot p_{{\mathcal{A} \cup \{i\}}}$ is 
\[
\frac{(a_{\ell+3}-a_\ell+1)!}{(i-a_\ell)! (a_{\ell+3}-i)!}
\cdot \left( \prod_{1 \leq s \leq m-1, \, s \textup{ odd
      and }
    s\neq \ell, \ell+2} \frac{1}{(a_{s+1}-a_s
      +1)!} \right) \cdot \left(\frac{1}{(a_{\ell+3} - a_\ell + 1)!} \right)
  \cdot \left( \prod_{j \in \mathcal{A} \cup \{i\}} p_j \right).
\]
Since by assumption on $i$ we have $i=a_{\ell+1}+1 = a_{\ell+2}-1$, we
obtain the simplification 
\begin{equation}
\begin{split}
\frac{(a_{\ell+3}-a_\ell+1)!}{(i-a_\ell)! (a_{\ell+3}-i)!}
\left(\frac{1}{(a_{\ell+3} - a_\ell + 1)!} \right) & =
\frac{1}{(i-a_\ell)! (a_{\ell+3}-i)!} \\
 & = \frac{1}{(a_{\ell+1}-a_\ell+1)!} \cdot \frac{1}{(a_{\ell+3} -
   a_{\ell+2} +1)!}
\end{split}
\end{equation}
from which it follows that $q_{i,\mathcal{A}} = 0$ also in this case.
The result follows.

  \end{proof}

\begin{example}
Let $n=5$, $i=4$ and let $\mathcal{A} = \{1,2\}$. 
Consider 
\[
m_{4, \{1,2\}} = p_4 \cdot p_{v_{\{1,2\}}} - c_{4,\{1,2\}}^{\{1,2,4\}} \cdot p_{v_{\{1,2,4\}}}.
\]
From~\eqref{eq:cBiA-formula-part1} it follows that $c_{4,\{1,2\}}^{\{1,2,4\}} =1$. The corresponding $q_{4, \{1,2\}}$ can be computed to be 
\[
q_{4, \{1,2\}} = p_4 \left( \frac{1}{2!} p_1 p_2 \right) - \left( \frac{1}{2!} p_1 p_2
\right) p_4 = 0. 
\]
\end{example}

We may now prove Theorem~\ref{theorem:simplify Monk}.

\begin{proof}[Proof of Theorem~\ref{theorem:simplify Monk}]
  By Theorem~\ref{theorem:presentation} we know 
\[
H^*_{S^1}(Y) \cong \C[t, \{p_{\mathcal{A}}\}_{\mathcal{A}
  \subseteq \{1,2,\ldots, n-1\}}] \large/ \mathcal{J}
\]
where $\mathcal{J}$ is the ideal generated by the
relations~\eqref{eq:Monk-final} so we wish to prove
\[
\C[t,p_1, \ldots, p_{n-1}]/\mathcal{K} \cong \C[t, \{p_{\mathcal{A}}\}_{\mathcal{A}
  \subseteq \{1,2,\ldots, n-1\}}] \large/ \mathcal{J}.
\]
The content of the Giambelli formula (Theorem~\ref{theorem:giambelli})
is that the expressions 
\[
p_{\mathcal{A}} - \sigma(\mathcal{A}) \prod_{j \in \mathcal{A}}
p_j
\]
are elements of $\mathcal{J}$. Hence 
\begin{equation*}
  \begin{split}
    \mathcal{J} &= \bigg\langle m_{i,\mathcal{A}} \, \bigg\vert \, 1 \leq i \leq
    n-1, \mathcal{A} \subseteq \{1,2,\ldots,n-1\} \bigg\rangle + \bigg\langle p_{\mathcal{A}} - \sigma(\mathcal{A}) \prod_{j \in \mathcal{A}}
p_j \, \bigg\vert \, 1 \leq i \leq
    n-1, \mathcal{A} \subseteq \{1,2,\ldots,n-1\} \bigg\rangle \\
 & = \bigg\langle q_{i, \mathcal{A}} \, \bigg\vert \, 1 \leq i \leq
    n-1, \mathcal{A} \subseteq \{1,2,\ldots,n-1\} \bigg\rangle + \bigg\langle p_{\mathcal{A}} - \sigma(\mathcal{A}) \prod_{j \in \mathcal{A}}
p_j \, \bigg\vert \, 1 \leq i \leq
    n-1, \mathcal{A} \subseteq \{1,2,\ldots,n-1\} \bigg\rangle.
 \end{split}
\end{equation*}
We therefore have
\[
\frac{\C[t, \{p_{\mathcal{A}}\}_{\mathcal{A}
  \subseteq \{1,2,\ldots, n-1\}}]}{\mathcal{J}} \cong \frac{\C[t,p_1,
\ldots, p_{n-1}]}{\bigg\langle q_{i, \mathcal{A}} \, \bigg\vert \, 1 \leq i \leq
    n-1, \mathcal{A} \subseteq \{1,2,\ldots,n-1\} \bigg\rangle}
\]
but since $q_{i,\mathcal{A}} = 0$ if $i \not \in \mathcal{A}$ by
Lemma~\ref{lemma:some q are zero} we conclude 
\[
\bigg\langle q_{i, \mathcal{A}} \, \bigg\vert \, 1 \leq i \leq
    n-1, \mathcal{A} \subseteq \{1,2,\ldots,n-1\} \bigg\rangle = 
\bigg\langle q_{i, \mathcal{A}} \, \bigg\vert \, 1 \leq i \leq
    n-1, \mathcal{A} \subseteq \{1,2,\ldots,n-1\} \textup{ and } i
    \in \mathcal{A} \bigg\rangle
\]
from which the result follows. 
\end{proof}

We illustrate the theorem by an example.

\begin{example}
Let $n=4$ and $Y$ the Peterson variety in $\Flags(\C^4)$. Then the degree-$2$ multiplicative generators are $p_1,
p_2$, and $p_3$. Then the statement of Theorem~\ref{theorem:simplify
  Monk} yields a presentation of the equivariant cohomology ring of
$Y$ as 
\[
H^*_{S^1}(Y) \cong \C[t, p_1, p_2, p_3]/\mathcal{K}
\]
where $\mathcal{K}$ is the ideal generated by the following $12$ elements: 
\begin{gather*}
2\, p_1^2 - 2t \, p_1 - p_1 p_2, \\
2 \, p _2^2 - 2t \, p_2 - p_1 p_2 - p_2 p_3, \\
2 \, p_3^2 - 2t \, p_3 - p_2 p_3, \\
3 \, p_1^2 p_2 - 6t \, p_1 p_2 - p_1 p_2 p_3, \\
3 \, p_1 p_2^2 - 6t \, p_1 p_2 - 2 \, p_1 p_2 p_3, \\
2 \, p_1^2 p_3 - 2t \, p_1 p_3 - p_1 p_2 p_3, \\
2 \, p_1 p_3^2 - 2t \, p_1 p_3 - p_1 p_2 p_3, \\
3 \, p_2^2 p_3 - 6t \, p_2 p_3 - 2 \, p_1 p_2 p_3, \\
3 \, p_2 p_3^2 - 6t \, p_2 p_3 - p_1 p_2 p_3, \\
p_1^2 p_2 p_3 - 3t \, p_1 p_2 p_3, \\
p_1 p_2^2 p_3 - 4t \, p_1 p_2 p_3, \\
p_1 p_2 p_3^2 - 3t \, p_1 p_2 p_3.
\end{gather*}
This list is not minimal: for instance, one can immediately see the sixth
and seventh expressions in this list are multiples of the first and
third ones, so evidently are unnecessary for defining the ideal
$\mathcal{K}$. In fact, more is true: a Macaulay 2 computation shows that the ideal
$\mathcal{K}$ is in fact generated by just the \emph{quadratic}
relations, 
i.e. the first three elements in the above list. (We thank the
referee for pointing this out.) 
Note the original presentation given in
Theorem~\ref{theorem:presentation} uses $8$ generators and $24$
relations, so this discussion shows that our presentation indeed gives
a simplification of the description of the ring. 
\end{example}

\begin{remark}\label{remark:quadratic}
We thank the referee for the following comment. Based on our Giambelli
formula, 
Theorem~\ref{theorem:simplify Monk}, and the example of $n=4$
discussed above, it seems natural to conjecture that for any value of
$n$, the corresponding ideal $\mathcal{K}$ is generated by just the
quadratic relations. 
Using Macaulay 2, we have verified that the conjecture holds for a
range of small values of $n$, but we were unable to give a 
proof for the general case. If the conjecture is true, then it would be a very significant
simplification of the presentation of this ring 
and would lead to many interesting geometric and combinatorial
questions. 
\end{remark}

\section{Stirling numbers of the second kind}\label{sec:stirling}

In this section we prove that Stirling numbers of the
second kind appear in the multiplicative structure of the ring
$H^*_{S^1}(Y)$. We learned this result from H. Naruse and do not claim
originality, though the proof given is our own. 
The \textbf{Stirling number of the second kind}, which we denote
$S(n,k)$, counts the number of ways to partition a set of $n$ elements
into $k$ nonempty subsets (see e.g. \cite[Section 1.2.6]{Knu73}). 
For example, $S(3,2)$ is the
number of ways to put balls labelled $1$, $2$, and $3$ into two
identical boxes such that each box contains at least one ball. It is
then easily seen that $S(3,2)=3$. We have the following.

\begin{theorem}
  Fix $n$ a positive integer. Let $Y$ be the Peterson variety in
  $\Flags(\C^n)$ equipped with the action of the $S^1$
  in~\eqref{eq:def-circle}. For ${\mathcal A} \subseteq \{1,2,\ldots,
  n-1\}$, let $v_{\mathcal{A}}, p_{\mathcal{A}}$ be as in
  Theorem~\ref{theorem:presentation}.  The following equality holds in
  $H^*_{S^1}(Y)$ for any $k$ with $1 \leq k \leq n-1$:
\begin{equation}\label{eq:stirling}
p_{1}^{k}=\sum_{j=1}^{k} S(k,j) t^{k-j} \, p_{v_{[1,j]}}.
\end{equation}
\end{theorem}

\begin{proof}
We proceed by induction on $k$. Consider the base case
$k=1$. Then~\eqref{eq:stirling} becomes the equality
\[
p_{1}=S(1,1)p_{1}. 
\]
Here $S(1,1)$ is the number of ways to put $1$ ball into $1$ box, so
$S(1,1)=1$ and the claim follows. 

Now assume that~\eqref{eq:stirling} holds for $k$. We need to show
that it also holds for $k+1$, i.e., 
\[
p_{1}^{k+1}=\sum_{j=1}^{k+1} S(k+1,j) t^{k+1-j} \, p_{v_{[1,j]}}.
\]
By the inductive hypothesis 
this is equivalent to showing
that 
\begin{equation}\label{eq:inductive step}
\sum_{i=1}^{k}S(k,i)t^{k-i}p_{1}p_{v_{[1, i]}}=\sum_{j=1}^{k+1}S(k+1,j)t^{k+1-j}p_{v_{[1,j]}}.
\end{equation}
We now expand the left hand side using the Monk formula. For each $i$
it can be computed that 
\[
p_1 \, p_{v_{[1,i]}} = it \, p_{v_{[1,i]}} + p_{v_{[1,i+1]}}
\]
where we have used \cite[Lemma 6.4]{HarTym09} to compute
$p_1(w_{[1,i]})$. 
Therefore 
\begin{equation*}
 \begin{split}
    \sum_{i=1}^{k}S(k,i)t^{k-i}p_{1}p_{v_{[1, i]}} &= 
\sum_{i=1}^{k}S(k,i)t^{k-i} (it \, p_{v_{[1,i]}} + p_{v_{[1,i+1]}}) \\
& = \sum_{i=1}^k i \, S(k,i) t^{k+1-i} p_{v_{[1,i]}} + \sum_{i=1}^k
S(k,i) t^{k-i} p_{v_{[1,i+1]}} \\
&= S(k,1) t^k p_1 + \sum_{i=2}^k i \, S(k,i) t^{k+1-i} 
 p_{v_{[1,i]}} + \sum_{i=1}^{k-1} S(k, i) t^{k-i} p_{v_{[1,i+1]}} + S(k,k)
 p_{v_{[1,k+1]}} \\ 
 & = S(k,1) t^k p_1 + \sum_{i=2}^k i \, S(k,i) t^{k+1-i} 
 p_{v_{[1,i]}} + \sum_{i=2}^k S(k, i-1) t^{k+1-i} p_{v_{[1,i]}} + S(k,k)
 p_{v_{[1,k+1]}} \\
 & = S(k+1,1) t^k p_1 + \sum_{i=2}^k (i \, S(k,i) + S(k,i-1))
 t^{k+1-i} p_{v_{[1,i]}}  + S(k+1,k+1) p_{v_{[1,k+1]}} \\
 & = S(k+1,1) t^k p_1 + \sum_{i=2}^k S(k+1,j)
 t^{k+1-i} p_{v_{[1,i]}}  + S(k+1,k+1) p_{v_{[1,k+1]}} \\
 & = \sum_{j=1}^{k+1} S(k+1,j) t^{k+1-j} p_{v_{[1,j]}} 
  \end{split}
\end{equation*}
where we have used 
the recurrence relation for Stirling numbers (see e.g. \cite{Knu73})
\[
S(k+1, j) = j S(k, j) + S(k, j-1)
\]
and the fact that $S(k,1)=S(k,k)=S(k+1,1)=S(k+1,k+1)=1$ for any
$k$. The result follows. 

\end{proof}

\section{Stability of Peterson Schubert classes}\label{sec:stability}

We now observe that the Peterson Schubert classes
$\{p_{\mathcal{A}}\}$ for the Peterson varieties satisfy a
stability property for varying $n$, similar to that satisfied by the
classical equivariant Schubert classes.  This is an observation we
learned from H. Naruse; we do not claim originality. For this section only, for
a fixed positive integer $n$ we denote by $Y_n$ the Peterson variety
in $\Flags(\C^n)$.

Let $X_{w,n} \subseteq \Flags(\C^n)$ denote the \textbf{Schubert
  variety} corresponding to $w \in S_n$ in $\Flags(\C^n)$. By the
standard inclusion of groups $S_n \into S_{n+1}$, we may also consider $w$ to be an
element in $S_{n+1}$. Furthermore there is a natural
$T^n$-equivariant inclusion
$\iota_n: \Flags(\C^n) \into \Flags(\C^{n+1})$ induced by the inclusion of the
coordinate subspace $\C^n$ into $\C^{n+1}$. Then with respect to
$\iota_n$ the Schubert variety $X_{w,n}$ maps isomorphically onto the
corresponding Schubert variety $X_{w,n+1}$. 
Since the equivariant Schubert classes are
cohomology classes corresponding to the Schubert varieties, this
implies that for any $w \in S_n$ there exists an infinite sequence of Schubert
classes $\{\sigma_{w,m}\}_{m=n}^{\infty}$ which lift the classes $\sigma_{w,n} \in H^*_{T^n}(\Flags(\C^n))$, i.e.,
\begin{equation}\label{eq:maps of flags}
\xymatrix{
\cdots \ar[r] & H^*_{T^n}(\Flags(\C^{n+2})) \ar[r] & 
H^*_{T^n}(\Flags(\C^{n+1})) \ar[r] & H^*_{T^n}(\Flags(\C^n)) \\
\cdots \ar@{|->}[r] & \sigma_{w,n+2} \ar@{|->}[r] & \sigma_{w, n+1}
\ar@{|->}[r] & \sigma_{w,n}
}
\end{equation}
and furthermore for any $v \in S_n$ and any $m \geq n$, the restriction
$\sigma_{w,m}(v)$ is equal to $\sigma_{w,n}(v)$. 
The theorem below asserts that a similar statement holds for
Peterson Schubert classes. Observe that the inclusion $\iota_n: \Flags(\C^n)
\into \Flags(\C^{n+1})$ mentioned above also induces a natural inclusion
$j_n: Y_n \into Y_{n+1}$ since the principal nilpotent operator on
$\C^{n+1}$ preserves the coordinate subspace $\C^n$. Moreover, since
the central circle subgroup of $U(n,\C)$ acts trivially on
$\Flags(\C^n)$ for any $n$, the inclusion $j_n$ is equivariant with
respect to the $S^1$-actions on $Y_n$ and $Y_{n+1}$ given by the two
circle subgroups defined
by~\eqref{eq:def-circle} in $U(n,\C)$ and $U(n+1,\C)$
respectively. Thus there is a pullback homomorphism $j_n^*:
H^*_{S^1}(Y_{n+1}) \to H^*_{S^1}(Y_n)$ analogous to the map \(\iota_n:
H^*_{T^n}(\Flags(\C^{n+1})) \to H^*_{T^n}(\Flags(\C^n))\) above. We
have the following.

\begin{theorem}
For a positive integer $n$ let $Y_n$ denote the Peterson variety in
$\Flags(\C^n)$ equipped with
the natural $S^1$-action defined by~\eqref{eq:def-circle}. For $w \in
S_n$ let
$p_{w,n} \in H^*_{S^1}(Y_n)$ denote the Peterson Schubert class corresponding to $w$. Then the
natural inclusions $j_m: Y_m \into Y_{m+1}$ for $m \geq n$ induce a
sequence of homomorphisms $j_m^*: H^*_{S^1}(Y_{m+1}) \to H^*_{S^1}(Y_m)$ such
that $j_m^*(p_{w,m+1}) = p_{w,m}$, i.e., there exists a infinite
sequence of Peterson Schubert
classes $\{p_{w,m}\}_{m=n}^{\infty}$ which lift $p_{w,n} \in
H^*_{T^n}(\Flags(\C^n))$
\begin{equation}
\xymatrix{
\cdots \ar[r] & H^*_{S^1}(Y_{n+2}) \ar[r] & H^*_{S^1}(Y_{n+1}) \ar[r] 
 & H^*_{S^1}(Y_n) \\
\cdots \ar@{|->}[r] & p_{w,n+2} \ar@{|->}[r] & p_{w, n+1}
\ar@{|->}[r] & p_{w,n}
}
\end{equation}
and furthermore for any $v \in Y_n^{S^1}$ and any $m \geq n$, the
restriction $p_{w,m}(v)$ is equal to $p_{w,n}(v)$. 
\end{theorem}

\begin{proof} 
By naturality and the definition of Peterson Schubert classes $p_{w,n}
\in H^*_{S^1}(Y_n)$ as the images of $\sigma_{w,n}$, it is immediate
that~\eqref{eq:maps of flags} can be expanded to a commutative diagram 
\begin{equation}
\xymatrix{
\cdots \ar[r] & H^*_{T^n}(\Flags(\C^{n+2})) \ar[r]^{\iota_{n+1}^*} \ar[d] & 
H^*_{T^n}(\Flags(\C^{n+1})) \ar[r]^{\iota_n^*} \ar[d] & H^*_{T^n}(\Flags(\C^n))
\ar[d] \\
\cdots \ar[r] & H^*_{S^1}(Y_{n+2}) \ar[r]_{j_{n+1}^*} & H^*_{S^1}(Y_{n+1}) \ar[r]_{j_n^*} 
 & H^*_{S^1}(Y_n)
}
\end{equation}
where the vertical arrows are the projection maps
$H^*_{T^n}(\Flags(\C^m)) \to H^*_{S^1}(Y_n)$ obtained by the
composition of $H^*_{T^n}(\Flags(\C^m)) \to H^*_{S^1}(\Flags(\C^m))$
with $H^*_{S^1}(\Flags(\C^m)) \to H^*_{S^1}(Y_n)$. In particular, for
any $w \in S_n$ and $m \geq n$, the 
vertical maps send $\sigma_{w,m}$ to $p_{w,m}$. The result follows.
\end{proof} 

\def\cprime{$'$}

\end{document}